\documentclass[12pt,reqno,intlim]{amsart}
\usepackage{amssymb,amsmath}

\usepackage[arrow,matrix,curve]{xy}

\newdir{ >}{{}*!/-5pt/\dir{>}}

\newcommand{\nc}{\newcommand}

\nc{\bC}{\bold{C}} \nc{\bN}{\Bbb{N}} \nc{\cF}{\mathcal{F}}
\nc{\cE}{\mathcal{E}} \nc{\cR}{\mathcal{R}} \nc{\cM}{\mathcal{M}}
\nc{\al}{\alpha} \nc{\bt}{\beta} \nc{\gm}{\gamma} \nc{\dl}{\delta}
\nc{\om}{\omega} \nc{\sg}{\sigma} \nc{\Sg}{\Sigma} \nc{\vf}{\varphi}
\nc{\ve}{\varepsilon} \nc{\os}{\overset} \nc{\ol}{\overline}
\nc{\ul}{\underline} \nc{\us}{\underset} \nc{\sbs}{\subset}
\nc{\bsl}{\backslash} \nc{\Ra}{\Rightarrow}
\nc{\lra}{\longrightarrow} \nc{\all}{\allowdisplaybreaks}

\nc{\Codes}{\operatorname{{\bold{Codes}}}}
\nc{\RegMono}{\operatorname{\mathcal{R}{\rm{eg}\mathcal{M}{\rm{ono}\!}}}}
\nc{\RegEpi}{\operatorname{\mathcal{R}{\rm{eg}\mathcal{E}{\rm{pi}\!}}}}
\nc{\Mn}{\operatorname{\mathcal{M}{\rm{ono}\!}}}
\nc{\Ep}{\operatorname{\mathcal{E}{\rm{pi}\!}}}
\nc{\Rg}{\operatorname{\mathcal{R}{\rm{eg}\!}}}
\nc{\Ob}{\operatorname{Ob\!}}

\numberwithin{equation}{section}

\newtheorem{theo}{\ \ \ Theorem}[section]
\newtheorem{lem}[theo]{\ \ \ Lemma}
\newtheorem{prop}[theo]{\ \ \ Proposition}

\theoremstyle{definition}

\theoremstyle{remark}

\begin{document}

\title[]
{Effective codescent morphisms of $n$-quasigroups and $n$-loops}

\author{Dali Zangurashvili}

\maketitle

\begin{abstract}
Effective codescent morphisms of $n$-quasigroups and of $n$-loops are characterized. To this end, it is proved that, for any $n\geq 1$, every codescent morphism of $n$-quasigroups (resp. $n$-loops) is effective. This statement generalizes our earlier results on qusigroups and loops. Moreover, it is shown that the elements of the amalgamated free products of $n$-quasigroups (resp. $n$-loops) have unique normal forms, and that the varieties of $n$-quasigroups and $n$-loops satisfy the strong amalgamation property. The latter two statements generalize the corresponding old results on quasigroups and loops by Evans.
\bigskip


\end{abstract}
\vskip+2mm

\section{Introduction}
The notion of an effective (co)descent morphism in a category is one of the main notions of the Grothendieck descent theory. The problem of characterizing effective codescent morphisms in varieties of universal algebras was posed by Janelidze. This problem was studied in a number of works \cite{M}, \cite{JT2}, \cite{Z1}, \cite{Z2}, \cite{Z3}, \cite{SZ}, \cite{Z4}, \cite{SZ1}. In \cite{Z3}, we proved that every codescent morphism of quasigroups (resp. loops) is effective. In the present paper, we generalize this result to the case of $n$-quasigroups and $n$-loops, for arbitrary $n\geq 1$.  This  generalization together with the criterion for a monomorphism of a regular category with pushouts and the strong amalgamation property to be a codescent morphism, found in \cite{Z1}, gives the characterization of effective codescent morphisms of $n$-quasigroups. 

We also show that, for any $n\geq 1$, the elements of the amalgamated free products of $n$-quasigroups (resp. $n$-loops) have unique normal forms and that the varieties of $n$-quasigroups and $n$-loops satisfy the strong amalgamation property. These statements generalize the corresponding old results on quasigroups and loops by Evans.
 
The author gratefully acknowledges the financial support from Shota Rustaveli National Science Foundation of Georgia (FR-22-4923).

 \section{The convergent representations of the varieties of n-quasigroups and $n$-loops}
 
We follow the notation used in the treatise \cite{B} by Belousov. Namely, for natural numbers $i$ and $j$, the symbol $x_i^j$ denotes the sequence $x_i,x_{i+1},...,x_{j}$ if $i<j$; it denotes the symbol $x_i$ if $i=j$, and denotes the empty sequence if $i>j$. The symbol $\stackrel{m}{e}$ denotes the sequence $e,e,...,e$ ($m$ times) if $m\geq 1$, and denotes the empty sequence if $m=0$.

Let $n$ be a natural number. Recall \cite{B} that an $n$-\textit{quasigroup} is defined as a set $Q$ equipped with an $n$-ary operation $f$ such that, for arbitrary $a_1,a_2,...,a_n\in Q$ and $i$ $(1\leq i\leq n)$, there is a unique element $b\in Q$ with 
\begin{equation}
f(a_1^{i-1},b,a_{i+1}^n)=a_i.
\end{equation}

An $n$-quasigroup $Q$ is called an $n$-\textit{loop} \cite{B} if there is an element $e\in Q$ such that, for any $i$ $(1\leq i\leq n)$, the identity
\begin{equation}
f(\stackrel{i-1}{e},x,\stackrel{n-i}{e})=x
\end{equation}
\noindent is satisfied in $Q$.

One can easily verify that the category of $n$-quasigroups (with morphisms being mappings preserving $f$) is the variety of universal algebras where the signature consists of the $n$-ary operation symbols $f, g_i$, while the identities are
\begin{equation}
f(x_1^{i-1},g_i(x_1^n),x_{i+1}^n)=x_i.
\end{equation}
\noindent and
\begin{equation}
g_i(x_1^{i-1},f(x_1^n),x_{i+1}^n)=x_i
\end{equation}
\noindent $(1\leq i\leq n)$. Similarly, the category of $n$-loops is the variety of universal algebras where the signature, in addition to the $n$-ary operation symbols $f, g_i$, contains one nullary operation $e$, while the identities are (2.2)-(2.4) $(1\leq i\leq n)$.
 
Obviously, $2$-quasigroups are precisely quasigroups, and $2$-loops are precisely loops. Observe also that $1$-quasigroups are precisely sets equipped with permutations, while $1$-loops are just sets.\vskip+1mm

Let $\mathcal{C}$ be a category with pushouts, and $p: B\rightarrow E$ be its morphism. The morphism $p$ induces the change-of-cobase functor $$p_{*}: B/\mathcal{C}\rightarrow E/\mathcal{C}$$ between the coslice categories; it sends a morphism $\varphi: B\rightarrow C$ to the pushout of $\varphi$ along $p$. As is well-known, this functor has a right adjoint. Recall that $p$ is called a \textit{codescent morphism} (resp. an \textit{effective codescent morphism}) if the functor $p_{*}$ is precomonadic (resp. comonadic) \cite{JT}.

In \cite{Z1}, we gave the necessary and sufficient condition for a morphism of a variety of universal algebras with the strong amalgamation property to be a codescent morphism. In \cite{Z3}, we gave the sufficient condition for any codescent morphism of a variety of universal algebras with the amalgamation property to be effective. Applying this sufficient condition, in the joint paper with Samsonadze \cite{SZ}, we related the problem of characterizing effective codescent morphisms in varieties of universal algebras to the notion of confluency that is one of the central notions in term rewriting theory. Below we recall some definitions from this theory (for more details, we refer the reader to the treatise \cite{BN} by Baader and Nipkow).

Let $\mathfrak{F}$ be a signature, i.e., a set of operation symbols equipped with arities. Let $X$ be a countable set with $\mathfrak{F} \cap X=\varnothing$. We assume that the reader is familiar with the notions of a $\mathfrak{F}$-term (or simply a term) over $X$, the size of a term, a position in a term, and the subterm of a term $t$ at a position $p$. We denote the latter subterm by $t\vert_p$, and denote the term obtained from $t$ by replacing the subterm $t\vert_p$ by a term $t'$ by the symbol $t[t']_p$. 

As usual, we identify positions in a term with nodes in the corresponding tree, and use the well-known rule of enumerating them with strings of natural numbers. For instance, the tree and the enumeration of nodes for the term $t=f(x_1,x_2,g_3(x_1,x_2,x_3,x_4),x_4)$ is the following:
$$
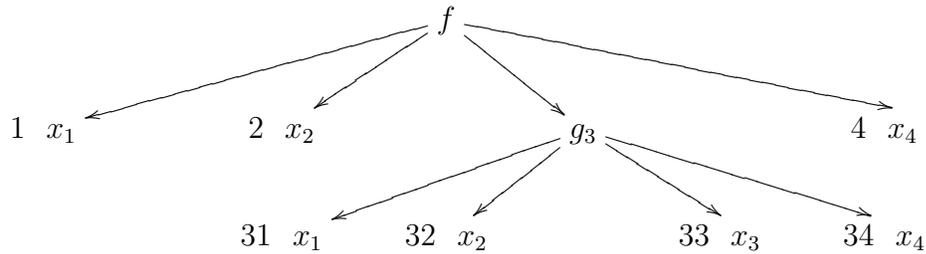
\xymatrix{&&&f\ar[dlll]\ar[dl]\ar[dr]\ar[drrr]!!<-2ex>&&&\\
1\,\,\,\,x_1&&2\,\,\,\,x_2&&g_3\ar[dll]\ar[dl]\ar[dr]!<1ex>\ar[drr]!<1ex>&&4\,\,\,\,x_4\\
&&31\,\,\,\,x_1&32\,\,\,\,x_2&&33\,\,\,\,x_3&34\,\,\,\,x_4}$$

\noindent If, for instance, $p=(3)$, then $t\vert_p=g_3(x_1,x_2,x_3,x_4)$. The node $f$ corresponds to the root position.

Let $T(\mathfrak{F},X)$ be the set of all terms over $X$, and $\sigma$ be a \textit{substitution} (i.e., a mapping $\sigma:X\rightarrow T(\mathfrak{F},X)$ such that the set $\lbrace x\mid \sigma(x)\neq x\rbrace$ is finite). Since $T(\mathfrak{F},X)$ has the structure of an $\mathfrak{F}$-algebra and is free over the set $X$, there is a unique $\mathfrak{F}$-homomorphism $\widehat{\sigma}:T(\mathfrak{F},X)\rightarrow T(\mathfrak{F},X)$ such that its restriction on $X$ is $\sigma$.

An \textit{oriented identity} is a formal expression of the form $l=r$, where $l$ and $r$ are terms (the order of the terms matters). In this paper, when no confusion might arise, \textbf{`identity' means an oriented identity.}

We denote  the set of variables that occur in at least one of terms $t_1,t_2,...,t_n$ by $Var(t_1,t_2,...,t_n)$. 

Let $\Sigma$ be a set of identities. A pair $(\mathfrak{F}, \Sigma)$ is called a \textit{term rewriting system} if $l$ is not a variable and $Var(r)\subseteq Var(l)$, for any identity $l=r$ from $\Sigma$. 

Let $(\mathfrak{F},\Sigma)$ be a term rewriting system. One introduces the following binary relation $\rightarrow$ on the set $T(\mathfrak{F},X)$ of terms: $t\rightarrow t'$ if the condition (C) below is satisfied. \vskip+2mm

(C) there exists an identity $l=r$ from $\Sigma$, a substitution $\sigma$ and a position $p$ of $t$ such that $t\vert_p=\widehat{\sigma}(l)$ and $t'=t[\widehat{\sigma}(r)]_p$.
\vskip+2mm
Roughly speaking, $t\rightarrow t'$ means that the term $t'$ can be obtained from $t$ by replacing some subterm, and this replacement is compliant with some identity from $\Sigma$. At that, we are permitted to use identities only in one direction -- from left to right.

We use the symbol $\xrightarrow{*}$ for the reflexive transitive closure of $\rightarrow$.
 
A term rewriting system $(\mathfrak{F}, \Sigma)$ (or simply $\Sigma$) is called \textit{terminating} if there is no infinite sequence of terms $$t_1\rightarrow t_2\rightarrow ...$$ A term rewriting system $(\mathfrak{F}, \Sigma)$ (or simply $\Sigma$) is called \textit{confluent} if, for any terms $t, t_1, t_2$ with $t_{1}\xleftarrow{*} t \xrightarrow{*} t_2$, there is a term $t'$ such that $t_1\xrightarrow{*} t' \xleftarrow{*} t_2$ (in that case, the terms $t_1$ and $t_2$ are called \textit{joinable}):
$$
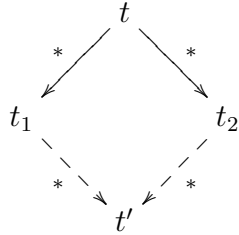
\xymatrix{&t\ar[dl]_{*}\ar[rd]^{*}&\\
t_1\ar@{-->}[dr]_{*}&&t_2\ar@{-->}[dl]^{*}\\
&t'&}$$
 
 Further, recall that a \textit{unifier} of terms $t$ and $t'$ is defined as a substitution $\sigma$ such that $\widehat{\sigma}(t)=\widehat{\sigma}(t')$. A unifier $\sigma$ is called \textit{most general} if any unifier of these terms can be obtained by composing $\widehat{\sigma}$ with some substitution. If two terms have at least one unifier, then they have also the most general unifier; it is unique up to \textit{renaming} (i.e. an injective substitution $\varrho$ with $\varrho(X)\subseteq X$) \cite{BN}.
 
 Let 
 $\iota_1:l_1=r_1$ and $\iota_2:l_2=r_2$ 
 be identities from $\Sigma$. Let us rename their variables so that
$$ Var(l_1,r_1)\cap Var(l_2,r_2)=\varnothing.$$
Let $p$ be a position of $l_1$ such that $l_1\vert_p$ is not a variable and the terms $l_1\vert_p$ and $l_2$ unify (in that case, one says that the term $l_1$ \textit{overlaps} the term $l_2$ at the position $p$). Let $\sigma$ be the most general unifier of $l_1\vert_p$ and $l_2$. The pair of terms $$(\widehat{\sigma}(r_1), \widehat{\sigma}(l_1)[\widehat{\sigma}(r_2)]_p)$$ is called the \textit{critical pair} determined by the identities $\iota_1$ and $\iota_2$ (at the position $p$ of $l_1$). 

Observe that the term  $\widehat{\sigma}(l_1)$ is in the relation $\rightarrow$ with both terms $\widehat{\sigma}(r_1)$ and  $\widehat{\sigma}(l_1)[\widehat{\sigma}(r_2)]_p$ participating in the critical pair:
\begin{equation}
\xymatrix{&\widehat{\sigma}(l_1)\ar[dl]\ar[dr]&\\
\widehat{\sigma}(r_1)&&\widehat{\sigma}(l_1)[\widehat{\sigma}(r_2)]_p}
\end{equation}

\noindent Moreover, the latter two terms are obviously joinable if $p$ is a root position of $\iota_1$ and $\iota_2$ is a renamed copy of $\iota_1$.
  
 \begin{theo}\cite{BN}
 Let $(\mathfrak{F},\Sigma)$ be terminating. Then it is confluent if and only if all critical pairs arisen from the identities of $\Sigma$ (including the critical pairs determined by identities and their renamed copies) are joinable.
 \end{theo}
 
\begin{lem}
The system of identities (2.3)-(2.4) is confluent if and only if $n=1$. 
\end{lem}

\begin{proof}
First observe that for any $i$ and $j$ ($1\leq i,j\leq n$), the left-hand term $l_1$ of identity (2.4) overlaps the left-hand term $l_2$ of the identity 
\begin{equation}
f(y_{1}^{j-1},g_j(y_1^{n}),y_{j+1}^{n})=y_j
\end{equation}
\noindent at the position $p=(i)$. 
The most general unifier $\sigma$ of $l_1\vert_p$ and $l_2$ is given as follows: $x_j\mapsto g_j(y_1,y_2,...,y_n)$ and $x_k\mapsto y_k$, for $k\neq j$. 
Assume now that $1\leq i<j\leq n$. 
Then diagram (2.5) takes the form

$$\xymatrix{&g_i(y_1^{i-1},f(y_1^{j-1},g_j(y_1^n),y_{j+1}^n),y_{i+1}^{j-1},g_j(y_1^n),y_{j+1}^n)\ar[dl]\ar[d]!<3ex>&\\
\,\,\,\,\,\,\,\,y_i&\,\,\,\,\,\,\,\,\,\,\,\,\,\,\,\,\,\,\,\,\,\,\,\,\,\,\,\,\,\,\,\,\,\,\,\,\,\,\,\,\,\,\,\,\,\,\,\,\,\,\,\,\,\,\,\,\,\,\,\,\,\,\,\,\,\,\,\,\,\,\,\,\,\,\,\,\,\,\,\,\,\,\,\,g_i(y_1^{i-1},y_{j},y_{i+1}^{j-1},g_j(y_1^n),y_{j+1}^n)}$$
\noindent
\vskip+2mm
 \noindent The terms in the obtained critical pair obviously are not joinable with respect to the system (2.3)-(2.4) of identities. Theorem 2.1 implies that this system is not confluent. The "if" part of the claim is immediate. 
 \end{proof}
 
 The arguments given in the proof of Lemma 2.2. imply that, in the variety of $n$-quasigroups with $n\geq 2$, the identity
 \begin{equation}\label{2.7}
 g_i(x_1^{i-1},x_{j},x_{i+1}^{j-1},g_j(x_1^n),x_{j+1}^n)=x_i
\end{equation}
is satisfied if $i<j$, and the identity
\begin{equation}\label{2.8}
g_i(x_1^{j-1},g_j(x_1^n),x_{j+1}^{i-1},x_{j},x_{i+1}^n)=x_i
 \end{equation}
is satisfied if $i>j$. 

\begin{lem}
Let $n\geq 2$. The system of identities (2.3)-(2.4) with $1\leq i,j\leq n$, (2.7) with $1\leq i<j\leq n$ and (2.8) with $1\leq j<i\leq n$ is a confluent representation of the variety of $n$-quasigroups. 
\end{lem}

\begin{proof}
To account all critical pairs, note that the left-hand terms $t_1$ and $t_2$ of identities (2.4) and (2.7) do not unify, for any $i<j$. Indeed, if $\sigma$ were their unifier, then the size $k_1$ of the subterm $t\vert_{(i)}$ of the term $t=\widehat{\sigma}(t_1)$ would be greater than the size $k_2$ of the subterm $t\vert_{(j)}$. On the other hand, $k_1<k_2$ since $t=\widehat{\sigma}(t_2)$, and we arrive to the contradiction. 

Consider now, for example, the case where  $\iota_1$ is identity $(2.3)$, $p=(i)$, and  $\iota_2$ is the renamed copy of $(2.7)$ (where $x_i$'s are replaced by $y_i$'s; $i<j$). The most general unifier of $l_1\vert_p$ and $l_2$ sends $x_i$ to $y_j$, $x_j$ to $g_j(y_1,...,y_n)$, and sends all other $x_k$'s to $y_k$'s. It is easy to see that diagram (2.5) takes the form 
$$\xymatrix{&f(y_1^{i-1},g_i(y_1^{i-1},y_j,y_{i+1}^{j-1},g_j(y_1^{n}),y_{j+1}^{n}), y_{i+1}^{j-1}, g_j(y_1^{n}), y_{j+1}^{n})\ar[dl]!<-1ex>\ar[d]!<3ex>&\\
\,\,\,\,\,\,\,\,\,\,\,\,\,y_j&\,\,\,\,\,\,\,\,\,\,\,\,\,\,\,\,\,\,\,\,\,\,f(y_1^{j-1},g_j(y_1^n),y_{j+1}^n)}$$
\noindent We obtain the critical pair 
that obviously is joinable due to identity (2.3). Similarly, one can verify that all other critical pairs are joinable.



\end{proof}

We are going now to deal with the variety of $n$-loops ($n\geq 1$). 

\begin{lem}
(i) Let $n$ be an arbitrary natural number. The system of identities (2.2)-(2.4) $(1\leq i\leq n)$ is not confluent. 

(ii) Let $n>1$. The union of the system of identities given in Lemma 2.3 with the system of identities (2.2) ($1\leq i\leq n$) is not confluent. 
\end{lem}

\begin{proof}
Let $\iota_1$ be identity (2.2), $\iota_2$ be identity (2.3), and $p$ be the root position of $l_1$. The most general unifier of $l_1$ and $l_2$  sends $x_j$ to $e$, for $j\neq i$, and sends $x$ to $g_i(e^{i-1},y_i,e^{n-i})$. Then diagram (2.5) takes the form
$$\xymatrix{
&f(\stackrel{i-1}{e},g_i(\stackrel{i-1}{e},x_i,\stackrel{n-i}{e}),\stackrel{n-i}{e})\ar[ld]\ar[rd]&\\
g_i(\stackrel{i-1}{e},x_i,\stackrel{n-i}{e})&&x_i}$$

\noindent We see that the terms in the critical pair that arises from this overlap is not joinable with respect to the set (2.2)-(2.4). If $n>1$, these terms are not joinable also with respect to the system of identities mentioned in the claim (ii). 
\end{proof}

Equating the terms in the critical pair arisen in the proof of Lemma 2.4,  we arrive to the identity
\begin{equation}
g_i(\stackrel{i-1}{e},x,\stackrel{n-i}{e})=x.
\end{equation}
\noindent Therefore, it is satisfied in the variety of $n$-loops, for any $n\geq 1$ and any $i$ $(1\leq i\leq n$). 


\begin{lem}
The system of identities (2-2)-(2.4), (2.9) $(1\leq i\leq n)$ is confluent if and only if $n=1$. If $n>1$, then the union of of the system of identities given in Lemma 2.4(ii) with the system of identities (2.9) ($1\leq i\leq n$) is not confluent.
\end{lem}

\begin{proof}
For "only if" part of the claim, let $1\leq i<j\leq n$. Let $\iota_1$ be identity $(2.4)$, $p=(i)$, and $\iota_2$ be the identity
$$f(\stackrel{j-1}{e},y,\stackrel{n-j}{e})=y.$$
 The most general unifier of $l_1\vert_p$ and $l_2$ sends $x_j$ to $y$, and $x_k$ to $e$, for $k\neq j$. Therefore, diagram (2.5) takes the form
$$\xymatrix{
&g_i(\stackrel{i-1}{e},f(\stackrel{j-1}{e},y,\stackrel{n-j}{e}),\stackrel{j-i-1}{e},y,\stackrel{n-j}{e})\ar[ld]\ar[dr]&\\
e&&g_i(\stackrel{i-1}{e},y,\stackrel{j-i-1}{e},y,\stackrel{n-j}{e})}$$
The terms $e$ and $g_i(\stackrel{i-1}{e},y,\stackrel{j-i-1}{e},y,\stackrel{n-j}{e})$ obviously are not joinable with respect to any system of identities considered above. The "if" part of the claim is trivial.
\end{proof}

 In view of the proof of Lemma 2.5, consider the identity
\begin{equation}
g_i(\stackrel{i-1}{e},x,\stackrel{j-i-1}{e},x,\stackrel{n-j}{e})=e,
\end{equation}
\noindent for any $i< j$, and the identity
\begin{equation}
g_i(\stackrel{j-1}{e},x,\stackrel{i-j-1}{e},x,\stackrel{n-i}{e})=e,
\end{equation}
\noindent for any $j< i$. The arguments of the above-mentioned proof imply that these identities are satisfied in the variety of $n$-loops, for any $n\geq 2$. Similarly to Lemma 2.3, one can verify the following 
\begin{lem}
Let $n\geq 2$. The system of identities (2.2)-(2.4) with $1\leq i\leq n$, (2.7) with $1\leq i<j\leq n$, (2.8) with $1\leq j<i\leq n$, (2.9) with $1\leq i\leq n$, (2.10) with $1\leq i<j\leq n$, and (2.11) with $1\leq j<i\leq n$ is a confluent representation of the variety of $n$-loops.
\end{lem}
 
 \section{Effective codescent morphisms in the varieties of $n$-quasigroups and $n$-loops}

 In \cite{SZ}, we considered the following conditions on the signature $\mathfrak{F}$ and identities of a term rewriting system $(\mathfrak{F},\Sigma)$:\vskip+2mm
 
 (*) for any identity $l=r$ from $\Sigma$, no variable occurs in $r$ more often than in $l$, and moreover, the size of $l$ is greater than the size of $r$;\vskip+2mm
  
 (**) if the set $\mathfrak{F}_0$ of constants from $\mathfrak{F}$ is not empty, then, for any non-trivial algebra $A$ from the variety determined by $\mathfrak{F}$ and $\Sigma$, the mapping $\mathfrak{F}_0\rightarrow A$ sending a constant to its value in $A$ is injective;\vskip+2mm
 
 (***) for any identity $l=r$ from $\Sigma$, any subterm $l'$ of $l$ which is neither a variable nor a constant, we have $Var(l')=Var(l)$.
\vskip+2mm

 Before continue, recall some definitions. 
 
 Let $\mathcal{V}$ be the variety of universal algebras determined by $(\mathfrak{F},\Sigma)$. Let $I$ be a non-empty set, and $(m_i:B\rightarrowtail A_i)_{i\in I}$ be a family of $\mathcal{V}$-algebras with an amalgamated subalgebra $B$.  We assume that $A_i\cap A_j=B$, for $i\neq j$ (and also that $A_i\cap \mathfrak{F}=\emptyset$).

One can introduce the binary relation $\rightsquigarrow$ on the set of $\mathfrak{F}$-terms over the set $\underset{i\in I}\cup A_i$ as follows: $t\rightsquigarrow t'$ if either the above-mentioned condition (C) is satisfied (for $X=\underset{i\in I}\cup A_i$) or $t'$ can be obtained from $t$ by replacing some subterm -- a subterm such that all its variables (being elements of algebras) belong to one and the same algebra $A_i$ -- by the value of this subterm in $A_i$.

The condition (*) implies that any element $\alpha$ of the free product $A$ of $(A_i)_i\in I$ with the amalgamated subgroup $B$ can be written as an irreducible term  over the set $\underset{i\in I}\cup A_i$, i.e., a term $\tau$ such that $\tau\rightsquigarrow \tau'$ for no term $\tau'$  over the same set. The term $\tau$ is called a \textit{normal form} of the element $\alpha$ \cite{SZ}. In general, an element $\alpha$ may have more than one normal form.\vskip+2mm

Finally, recall that a variety is said to satisfy the amalgamation property if, for any pushout 
\begin{equation}
\xymatrix{A\ar[r]^{f}\ar[d]_{g}& B\ar[d]^{g'}\\
C\ar[r]_{f'}&D}
\end{equation}
\noindent with monomorphic $f$ and $g$, the homomorphisms $f'$ and $g'$ also are such. A variety is said to satisfy the strong amalgamation property if it satisfies the amalgamation property and, moreover, for any pushout (3.1) with monomorphic $f$ and $g$, one has $f'(C)\cap g'(B)=f'g(A)$ \cite{KMPT}.

 \begin{theo}\cite{SZ}
 Let a variety $\mathcal{V}$ of universal algebras be represented by a confluent term rewriting system $(\mathfrak{F},\Sigma)$ that satisfies the conditions (*)-(***). Then 
 
 (a) every codescent morphism of $\mathcal{V}$ is effective;
 
 (b) the elements of amalgamated free products in $\mathcal{V}$ have unique normal forms;
 
 (c) the variety $\mathcal{V}$ satisfies the strong amalgamation property.
 \end{theo}
 \vskip+1mm
 From Lemmas 2.2, 2.3, 2.5, 2.6 and Theorem 3.1 we obtain
 
 \begin{theo} Let $n$ be an arbitrary natural number.
 
 (a) Every codescent morphism of $n$-quasigroups (resp. $n$-loops) is effective;
 
 (b) the elements of amalgamated free products in the variety of $n$-quasigroups (resp. $n$-loops) have unique normal forms;
 
 (c) the variety of $n$-quasigroups (resp. $n$-loops) satisfies the strong amalgamation property.
 \end{theo}
 
 Note that, for $n=2$, the claim (a) of Theorem 3.2 was first given in \cite{Z3} (see also \cite{SZ}), while the claims (b) and (c) were first given in the paper \cite{E} by Evans (see also \cite{JK} and \cite{Z3}). For $n=1$, the claim (a) is obvious since the variety of  $1$-quasigroups (being isomorphic to the category $\mathcal{S}et^{\mathbb{Z}}$ of functors from the group $\mathbb{Z}$ of integers, viewed as a category, to the category of sets) and the variety of $1$-loops are topoi. The claim (c), for $1$-loops is obvious. For $1$-quasigroups, the latter claim immediately follows from the main result of \cite{Z5}.

\vskip+1mm 
 
 Recall 
 \begin{theo} \cite{JT1} Let $\mathcal{C}$ be a category with pushouts and equalizers. A morphism p is a codescent morphism if and only if it is a couniversal regular monomorphism. i.e. a morphism whose any pushout is
a regular monomorphism.
\end{theo}

It is easy to observe that any monomorphism is regular in a variety with the strong amalgamation property.
 
 \begin{theo} \cite{Z1}
 Let a variety $\mathcal{V}$ of universal algebras satisfy the strong amalgamation property. Then a monomorphism $A\rightarrowtail A'$ of $\mathcal{V}$ is a codescent morphism if and only if it satisfies the congruence extension property, i.e., for any congruence $R$ on $A$, there is a congruence $R'$ on $A'$ such that $R'\cap (A\times A)=R$.
 \end{theo}
 
 Note that congruences mentioned in the congruence extension property in Theorem 3.4 can be congruences with respect to any representation of the variety (while normal forms mentioned in Theorem 3.1(c) are ones with respect to the considered representation $(\mathfrak{F},\Sigma)$). Therefore, Theorem 3.2(a), Theorem 3.2(c) and Theorem 3.4 imply
 \begin{theo}
 Let $n$ be an arbitrary natural number. A monomorphism $Q\rightarrowtail Q'$ of $n$-quasigroups (resp. $n$-loops) is an effective codescent morphism if and only if it satisfies the $\mathfrak{F}$-congruence extension property, where $\mathfrak{F}=\lbrace f
 \rbrace \cup \lbrace g_i\mid 1\leq i\leq n\rbrace$ (and all operation symbols are $n$-ary).
 \end{theo}
 
 Note that, for $n=2$, the set of $\mathfrak{F}$-congruences on a quasigroup $Q$ do not coincide with that of $\lbrace f\rbrace$-congruences on $Q$ \cite{JK}. The former congruences are referred to as normal congruences in literature. Not any monomorphism of quasigroups (resp. loops) satisfies the normal-congruence extension property \cite{JK}. 
 \vskip+1mm
 
 Let $n=1$. Then any monomorphism is an effective codescent morphism in both the category of $1$-quasigroups and the category of $1$-loops since they are topoi. Theorem 3.5 implies that, in these varieties, any monomorphism satisfies the $\mathfrak{F}$-congruence extension property (this can be easily seen immediately).  
 Observe also that, not any $\lbrace f\rbrace$-congruence on a $1$-quasigroup is an $\mathfrak{F}$-congruence (for an example, consider a $1$-quasigroup $(\mathbb{Z},f)$ with $f(k)=k+1$, for any $k\in \mathbb{Z}$. Let $\theta$ be the equivalence relation on $(\mathbb{Z},f)$  with the following equivalence classes: $\lbrace k \rbrace$ ($k<1$), and $\mathbb{N}$. The relation $\theta$ is obviously a congruence with respect to the signature $\lbrace f\rbrace$, but is not a congruence with respect to the one  $\mathfrak{F}=\lbrace f, f^{-1}
 \rbrace$). Nevertheless, we have
 
 \begin{prop}
 Let $Q$ be a finite 1-quasigroup. Then any $\lbrace f\rbrace$-congruence on $Q$ is an $\mathfrak{F}$-congruence. 
 \end{prop}
 
 \begin{proof}
 Let $\theta$ be an $\lbrace f\rbrace$-congruence on $Q$. Let $C_a$ denote the $\theta$-class of an element $a\in Q$. One obviously has 
 \begin{equation}
 f(C_a)\subseteq C_{f(a)}.
 \end{equation}
 Inequality (3.2) implies that we have the increasing chain of natural numbers:
 $$card(C_a)\leq card(C_{f(a)})\leq card(C_{f^2(a)})\leq ...$$
 Since $Q$ is finite, this chain stabilizes at some step $k\geq 0$ (assuming that $f^0=id_Q$). Inequality (3.2) implies that, for any $m\geq k$, we have 
 \begin{equation}
 f(C_{f^m(a)})=C_{f^{m+1}(a)}.
 \end{equation}
 Assume now that $k\geq 1$. Since $Q$ is finite and no distinct $\theta$-classes have non-empty intersection, (3.3) implies that there is the smallest $l\geq k$ with 
 $$ f(C_{f^l(a)})=C_{f^l(a)}.$$
 Since $l\geq 1$ and $f$ is a bijection, we arrive to the contradiction. Therefore, $k=0$, and we have 
 $$f(C_a)=C_{f(a)}.$$
 \end{proof}

\vskip+2mm
 
\textit{Author's address:}

\textit{Dali Zangurashvili, A. Razmadze Mathematical Institute of Tbilisi State University, 6 Alexidze Str., 0193, Georgia;}

\textit{e-mail: dali.zangurashvili@tsu.ge}

\end{document}